\documentclass[reqno]{amsart}
\usepackage{amssymb}
\usepackage{graphicx}

\usepackage[usenames, dvipsnames]{color}
\usepackage{verbatim}
\usepackage{mathrsfs}
\usepackage{bm}
\usepackage{cite}
\usepackage{bbm}

\numberwithin{equation}{section}

\newtheorem{theorem}{Theorem}[section]
\newtheorem{corollary}[theorem]{Corollary}
\newtheorem{lemma}[theorem]{Lemma}
\newtheorem{prop}[theorem]{Proposition}

\theoremstyle{definition}
\newtheorem{remark}[theorem]{Remark}

\theoremstyle{definition}

\theoremstyle{definition}

\makeatletter
\def\dashint{\operatorname%
{\,\,\text{\bf-}\kern-.98em\DOTSI\intop\ilimits@\!\!}}
\makeatother

\def\\det{\text{\det}}

\def\Xint#1{\mathchoice
 {\XXint\displaystyle\textstyle{#1}}%
 {\XXint\textstyle\scriptstyle{#1}}%
 {\XXint\scriptstyle\scriptscriptstyle{#1}}%
 {\XXint\scriptscriptstyle\scriptscriptstyle{#1}}%
 \!\int}
\def\XXint#1#2#3{{\setbox0=\hbox{$#1{#2#3}{\int}$}
  \vcenter{\hbox{$#2#3$}}\kern-.5\wd0}}

\def\dashint{\Xint-}

\def\.5{\frac{1}{2}}

\newcommand{\RN}[1]{%
  \textup{\uppercase\expandafter{\romannumeral#1}}%
}

\newcommand{\set}[1]{\left\{#1\right\}}

\renewcommand{\epsilon}{\varepsilon}

\newcounter{marnote}

\begin{document}

\title[Local behavior of singular parabolic equations]{Local behavior of solutions to inhomogeneous singular parabolic $p$-Laplace equations}

\author[X. Hao]{Xia Hao}
\address[X. Hao] {School of Mathematical Sciences, Hebei Normal University, Shijiazhuang 050024, China.}
\email{xiahao0915@163.com }

\author[Y. Li]{Yan Li}
\address[Y. Li] {School of Mathematical Sciences, Beijing Normal University, Beijing 100875, China. }
\email{yanli25@mail.bnu.edu.cn}

\author[Z.W. Zhao]{Zhiwen Zhao}

\address[Z.W. Zhao]{School of Mathematics and Physics, University of Science and Technology Beijing, Beijing 100083, China.}
\email{zwzhao365@163.com;\,zwzhao365@gmail.com}

%\footnote{}

\date{\today} % delete this line to display the current date

%%% BEGIN DOCUMENT

\maketitle
%\tableofcontents
\begin{abstract}
It is known that a major difficulty in proving H\"{o}lder regularity for solutions to quasilinear singular parabolic equations of $p$-Laplace type via the method of intrinsic scaling is to establish the decay estimate of the space-time measure of level sets.\ In this paper, we consider an inhomogeneous singular parabolic $p$-Laplace equation with nonnegative time-independent forcing and Dirichlet data.\ By combining a time-rescaling argument with $L^\infty$ comparison estimates, we derive Lipschitz regularity in time after any fixed positive elapsed time, which reduces the problem to an elliptic-type decay estimate for the spatial measure of level sets that is essentially uniform in time. This reduction enables us to address the above obstacle.

\end{abstract}

\maketitle
%\date{}
%\maketitle
%{\bf Abstract}

%{\bf{Keywords:}} H\"{o}lder estimates; singular parabolic $p$-Laplace equations; temporal Lipschitz regularity
%
%\noindent{\bf{MSC numbers}}: {35K92; 35B51; 35B65.}

\section{Introduction}

Let $\Omega\subset \mathbb{R}^{n}$ be a smooth bounded domain with $n\geq2$.\ For $T>0$, set $\Omega_{T}=\Omega\times(-T,0]$, and denote by $\partial_{pa}\Omega_{T}$ the standard parabolic boundary of $\Omega_{T}$.\ For $p>1$, we consider the following inhomogeneous parabolic $p$-Laplace equation subject to the Dirichlet boundary condition:
\begin{align}\label{PO001}
\begin{cases}
\partial_{t}u-\mathrm{div}(|\nabla u|^{p-2}\nabla u)=f(x),& \mathrm{in}\;\Omega_{T},\\
u=\phi(x),&\mathrm{on}\;\partial_{pa}\Omega_{T},
\end{cases}
\end{align}
where $f,\phi\geq0$ satisfy
\begin{align}\label{AQ000}
\|f\|_{L^{\infty}(\Omega_{T})}=\|f\|_{L^{\infty}(\Omega)}<\infty,\quad\|\phi\|_{L^{\infty}(\partial_{pa}\Omega_{T})}=\|\phi\|_{L^{\infty}(\Omega)}<\infty.
\end{align}
A function $u\in C([-T,0];L^{2}(\Omega))\cap L^{p}((-T,0);W^{1,p}(\Omega))$ is called a weak solution to problem \eqref{PO001} if, for every
$-T\leq t_{1}<t_{2}\leq0$,
\begin{align*}
	&\int_{\Omega}u\varphi\,dx\Big|_{t_{1}}^{t_{2}}+\int_{t_{1}}^{t_{2}}\int_{\Omega}\left(-u\partial_{t}\varphi
	+|\nabla u|^{p-2}\nabla u\cdot\nabla\varphi\right)\,dxdt=\int_{t_{1}}^{t_{2}}\int_{\Omega}f(x)\varphi\,dxdt,
\end{align*}
for any $\varphi\in W^{1,2}((-T,0);L^{2}(\Omega))\cap L^{p}((-T,0);W^{1,p}_{0}(\Omega))$.

From an analytical point of view, when $p>2$, the parabolic $p$-Laplace equation is degenerate, since the modulus of ellipticity $|\nabla u|^{p-2}$ vanishes at points where $|\nabla u|=0$. In contrast, when $1<p<2$, the equation is singular, because $|\nabla u|^{p-2}$ becomes unbounded as $|\nabla u|$ approaches zero.\ The singularity and degeneracy of the equation prevent solutions from enjoying the full smoothness available in the linear case $p=2$ and, in general, restrict their regularity to $C^{1,\alpha}$ for some $0<\alpha<1$.\ A substantial body of literature has focused on the regularity for quasilinear elliptic and parabolic equations, see e.g.  \cite{BDGLS2026,L2019,D1986,CD1988,DZ2022,DZ2024,D1983,D1993,DM2010,DM2011,U2008,E1982,DK1992,L1994,DGV2012,DGV2008,DGV200802,AL1983,DF198501,DF198502,DF1984,C1991,BSS2022}.\ Related regularity theory has also been developed for general doubly nonlinear parabolic equations, see e.g.  \cite{I1989,I199401,I199402,I1995,BDMS2018,BDL2021,BDGLS202301,BDGLS202302,KSU2012,PV1993,VV2022,BFV2018,V2007,DK2007,DKV1991,JX2019,JRX2023,K1988}.\ In particular, \cite{BDGLS202302} provides a systematic introduction to this subject.

In the degenerate case $p>2$, DiBenedetto \cite{D1986,D1993} developed the intrinsic scaling technique to study H\"older regularity for solutions to general quasilinear parabolic equations.\ DiBenedetto, Gianazza, and Vespri  \cite{DGV2008,DGV2012} further introduced an exponential change of variables to overcome the obstacle caused by the expansion in time, thereby refining the method of intrinsic scaling.\ Especially, they established a stronger Harnack inequality that provides a finer description for the regularity of solutions. 

However, for the singular range $1<p<2$, the difficulty shifts from the time-expansion issue to establishing decay estimates for the space-time measure of level sets.\ In this paper, we solve this problem by proving temporal Lipschitz regularity away from the initial time, which is achieved by coupling a time-rescaling argument with $L^\infty$-stability estimates derived from the comparison principle.\ More importantly, beyond refining the intrinsic scaling method in the singular regime, our work further develops the parabolic De Giorgi theory to handle singular quasilinear parabolic equations of $p$-Laplace type.

Apart from its intrinsic mathematical interest, the parabolic $p$-Laplace equation arises in various applications, such as nonlinear porous-media flow, image processing, and game theory, see e.g. \cite{DLK2013,ACM2004,BV2004,PS2008}.\ According to the diffusion behavior, the equation can be divided into the following three regimes.\ The range $1<p<2$ corresponds to fast diffusion, characterized by finite-time extinction of solutions. By contrast, for $p>2$, the equation exhibits slow diffusion, in which solutions approach the steady state at an algebraic rate. At the borderline $p=2$, the equation reduces to the classical heat equation, whose solutions exhibit exponential decay in time.

Throughout the paper, $C$ denotes a positive constant depending only on the prescribed data including $n$,\ $p$,\ $\|\phi\|_{L^{\infty}(\Omega)}$,\ $\|f\|_{L^{\infty}(\Omega)}$, and its value may change from line to line.\ Our main result is now stated as follows.
\begin{theorem}\label{ZWTHM90}
Set $1<p<2$ and $n\geq2$. Suppose that $u$ is a weak solution of \eqref{PO001} with $\Omega\times(-T,0]=B_{1}\times(-1,0]$.\ Then there exists a constant $\alpha\in(0,\frac{p}{2}]$ depending only on the above data such that for every $(x,t),(y,s)\in B_{2^{-1}+2^{-2/p}}\times(-2^{-1},0],$
\begin{align}\label{QNAW001ZW}
|u(x,t)-u(y,s)|\leq C\big(|x-y|^{\alpha}+|t-s|\big).
\end{align}

\end{theorem}
\begin{remark}
The estimate in \eqref{QNAW001ZW} exhibits two different regularity mechanisms. The linear dependence on $|t-s|$ follows from the time-Lipschitz continuity established in Theorem \ref{thm10} below, whereas the spatial H\"{o}lder exponent $\alpha$ is obtained through the method of intrinsic scaling  refined in this paper for the singular case. 
\end{remark}

\begin{remark}
For any given compact subdomain $\Omega_{0}\subset\subset\Omega$ and $T>0$, set $R_{0}=\min\{1,\mathrm{dist}(\Omega_{0},\partial\Omega),T^{\frac{1}{p}}\}.$ Then applying the proof of Theorem \ref{ZWTHM90} with minor modification, we obtain that the H\"{o}lder estimate in \eqref{QNAW001ZW} holds with $B_{2^{-1}+2^{-2/p}}\times(-2^{-1},0]$ replaced by $\Omega_{0}\times(-T/2,0]$. In this case the constant $C$ may depend on $R_{0}$, but $\alpha$ is independent of $R_{0}$.

\end{remark}

\begin{remark}
	\label{REM-boundary-counterexample}
	A uniform spatial H\"older estimate on the entire ball cannot, in general,
	be obtained from bounded Dirichlet data alone, even when the data are
	nonnegative and time-independent. Indeed, fix $n=2$, $1<p<2$, and define
	\begin{align*}
		w(x_1,x_2)=\frac{\pi}{2}+\arctan\frac{x_2}{1-x_1},\quad (x_1,x_2)\in B_1.
	\end{align*}
	Since $1-x_1>0$ in $B_1$, we have $w\in C^\infty(B_1)$ and $0<w<\pi$.\ Writing $r^2=(1-x_1)^2+x_2^2$, a direct calculation gives
	\begin{align*}
		\nabla w=\frac{(x_2,1-x_1)}{r^2},
		\quad
		|\nabla w|=\frac1r,
		\quad
		|\nabla w|^{p-2}\nabla w=\frac{(x_2,1-x_1)}{r^p}.
	\end{align*}
	Consequently,
	\begin{align*}
		\mathrm{div}\big(|\nabla w|^{p-2}\nabla w\big)=0,\quad\mathrm{in}\;B_1.
	\end{align*}
	Moreover, the only singularity of the gradient lies at the boundary point
	$(1,0)$, and
	\begin{align*}
		\int_{B_1}|\nabla w|^pdx
		\leq 2\pi\int_0^2 s^{1-p}ds
		=\frac{2^{3-p}\pi}{2-p}<\infty.
	\end{align*}
	Thus $w\in W^{1,p}(B_1)\cap L^\infty(B_1)$.\ Since $|\nabla w|^{p-2}\nabla w\in L^{\frac{p}{p-1}}(B_1)$, the equation $-\mathrm{div}(|\nabla w|^{p-2}\nabla w)=0$ holds weakly for every test function in $W_0^{1,p}(B_1)$.
	
	Now set
	\begin{align*}
		u(x,t)=w(x),\quad f(x)=0,\quad \phi(x)=w(x).
	\end{align*}
Then $u\in C([-1,0];L^2(B_1))\cap L^p((-1,0);W^{1,p}(B_1))$ is a bounded nonnegative stationary weak solution of
	$\partial_{t}u-\mathrm{div}(|\nabla u|^{p-2}\nabla u)=0$ in $B_1\times(-1,0)$.
	It satisfies $u(\cdot,-1)=\phi$ in $L^2(B_1)$ and the lateral
	Dirichlet condition $u=\phi$ in the Sobolev trace sense.
	In particular, $\partial_tu\equiv0$.
	
	For $0<\varepsilon<1/2$, take
	\begin{align*}
		x_\varepsilon=(1-\varepsilon,\varepsilon),
		\quad
		y_\varepsilon=(1-\varepsilon,-\varepsilon).
	\end{align*}
	Both points belong to $B_1$, since
	\begin{align*}
		|x_\varepsilon|^2=|y_\varepsilon|^2
		=1-2\varepsilon+2\varepsilon^2<1.
	\end{align*}
	However,
	\begin{align*}
		|x_\varepsilon-y_\varepsilon|=2\varepsilon,
		\quad
		w(x_\varepsilon)=\frac{3\pi}{4},
		\quad
		w(y_\varepsilon)=\frac{\pi}{4}.
	\end{align*}
	Hence, at any fixed time $t_0\in(-1/2,0]$, a uniform estimate
	\begin{align*}
		|u(x,t)-u(y,s)|
		\leq C\big(|x-y|^\alpha+|t-s|\big),
		\quad
		(x,t),(y,s)\in B_1\times(-1/2,0],
	\end{align*}
	with $\alpha>0$ and $C<\infty$, would imply
	\begin{align*}
		\frac{\pi}{2}
		=|u(x_\varepsilon,t_0)-u(y_\varepsilon,t_0)|
		\leq C(2\varepsilon)^\alpha,
	\end{align*}
	which is impossible as $\varepsilon\rightarrow0$.
	
Therefore, without additional spatial regularity of the Dirichlet data, such an estimate cannot hold uniformly up to the lateral boundary.\ The obstruction is purely spatial and persists at every positive elapsed time.\ It does not contradict interior regularity: for every $0<\rho<1$,
	\begin{align*}
		\|\nabla u\|_{L^\infty(B_\rho)}=\|\nabla w\|_{L^\infty(B_\rho)}
		\leq\frac1{1-\rho}.
	\end{align*}
\end{remark}

%\begin{remark}
%Since our results cover the case of $p=2$, it doesn't need to analyze the stability for $p$ near $2$ any more.
%
%\end{remark}

%Especially our proof holds the promise of a wide application to more types of quasilinear parabolic equations.

The remainder of this paper is structured as follows.\ In Section \ref{SEC002}, we collect some preliminary results.\ In Section \ref{SEC003}, we derive a global $L^{\infty}$ bound for weak solutions. Section \ref{SEC005} is devoted to establishing an essentially uniform-in-time spatial distribution decay estimate.\ Finally, Section \ref{SEC006} provides the remaining two ingredients for the parabolic De Giorgi truncation method, namely, the measure-to-point oscillation reduction estimate and the expansion in time, thereby completing the proof of Theorem \ref{ZWTHM90}.

\section{Preliminary}\label{SEC002}
We begin by fixing some notation to be used throughout the paper. For $x_{0}\in\mathbb{R}^{n}$ and $\rho>0$, let $B_{\rho}(x_{0})$ denote the ball of radius $\rho$ centered at $x_{0}$.\ For $t_{0}\leq 0$ and $\tau>0$, the backward parabolic cylinder is defined by
\begin{align*}
[(x_{0},t_{0})+Q(\rho,\tau)]=B_{\rho}(x_{0})\times(t_{0}-\tau,t_{0}].
\end{align*}
For notational convenience, we set $B_{\rho}=B_{\rho}(0)$ and $Q(\rho,\tau)=[(0,0)+Q(\rho,\tau)]$. 

We now recall the isoperimetric inequality and the Sobolev embedding theorem that will be used below. These two tools are essential ingredients in the
De Giorgi truncation method.
\begin{lemma}[see Lemma 2.2 in Chapter 2 of \cite{DGV2012}]
For $n\geq2$ and $1<p<\infty$, we obtain that for any $R>0$, $l>k$ and $u\in W^{1,1}(B_{R})$,
\begin{align}\label{pro001}
&(l-k)|\{u\geq l\}\cap B_{R}|\leq C(n,p)\frac{R^{n+1}}{|\{u\leq k\}\cap B_{R}|}\int_{\{k<u<l\}\cap B_{R}}|\nabla u|dx.
\end{align}
\end{lemma}
\begin{lemma}[see Proposition 4.1 in Chapter 2 of \cite{DGV2012}]\label{lem001}
Set $n\geq2$ and $1<p\leq2$.\ For any $B_{R}(x_{0})\subset \mathbb{R}^{n}$ and $u\in L^{2}(B_{R}(x_{0}))\cap W_{0}^{1,p}(B_{R}(x_{0}))$, we have
\begin{align*}
\int_{B_{R}(x_{0})}|u|^{p\tilde{\chi}}dx\leq C(n,p)\int_{B_{R}(x_{0})}|\nabla u|^{p}dx\bigg(\int_{B_{R}(x_{0})}|u|^{2}dx\bigg)^{\chi-1},
\end{align*}
where $\chi$ and $\tilde{\chi}$ are defined by
\begin{align}\label{chi}
\chi=\frac{n+p}{n},\quad \tilde{\chi}=\frac{n+2}{n}.
\end{align}

\end{lemma}

We end this section by giving the following dual strong monotonicity inequality for the $p$-Laplace vector field.
\begin{lemma}\label{Lem0006}
	Let $n\geq2$ and $1<p<2$. For any $a,b\in\mathbb{R}^{n}$, we have
	\begin{align}\label{DP02}
		\big(|a|^{p-2}a-|b|^{p-2}b\big)\cdot(a-b)\geq2^{\frac{p-2}{p-1}}||a|^{p-2}a-|b|^{p-2}b|^{\frac{p}{p-1}}.
	\end{align}

\end{lemma}

\begin{proof}
For $p>1$ and $a\in\mathbb{R}^{n}$, denote
\begin{align*}
\mathcal{K}_{p}(a)=|a|^{p-2}a,\quad 	p'=\frac{p}{p-1}.
\end{align*}	
From Chapter 12 in \cite{L2019}, we recall the strong monotonicity inequality as follows: for $a,b\in\mathbb{R}^{n},$ 
\begin{align}\label{DP01}
\big(\mathcal{K}_{q}(a)-\mathcal{K}_{q}(b)\big)\cdot(a-b)\geq2^{2-q}|a-b|^{q},\quad\text{if $q\geq2$.}
\end{align}
Observe that 
\begin{align*}
\mathcal{K}_{p'}(\mathcal{K}_{p}(a))=a,\quad\text{for any $p>1$ and $a\in\mathbb{R}^{n}$}.	
\end{align*}	
Therefore, applying \eqref{DP01} with $q=p'$, we deduce that for $1<p<2$ and $a,b\in\mathbb{R}^{n}$,
	\begin{align*}
		\big(\mathcal{K}_{p}(a)-\mathcal{K}_{p}(b)\big)\cdot(a-b)=&\big(\mathcal{K}_{p}(a)-\mathcal{K}_{p}(b)\big)\cdot(\mathcal{K}_{p'}(\mathcal{K}_{p}(a))-\mathcal{K}_{p'}(\mathcal{K}_{p}(b)))\notag\\
		\geq& 2^{2-p'}|\mathcal{K}_{p}(a)-\mathcal{K}_{p}(b)|^{p'}.
	\end{align*}
	This proves \eqref{DP02}.
\end{proof}

\section{Global boundedness of weak solutions}\label{SEC003}
For $k\in\mathbb{R}$ and $u\in C([-1,0];L^{2}(B_{1}))\cap L^{p}((-1,0);W^{1,p}(B_{1}))$, denote
\begin{align*}
	(u-k)_{+}=\max\{u-k,0\},\quad(u-k)_{-}=\max\{k-u,0\}.
\end{align*}
To prove the global $L^{\infty}$ estimate, we first establish the following energy inequality.
\begin{lemma}\label{lem003}
Assume that $u$ is the solution of \eqref{PO001} with $\Omega_{T}=B_{1}\times(-1,0]$. For $k\in\mathbb{R}$, set $v_{\pm}=(u-k)_{\pm}$, and let $[(x_{0},t_{0})+Q(\rho,\tau)]\subset B_{1}\times(-1,0]$.\ Then for any $\xi\in C^{\infty}([(x_{0},t_{0})+Q(\rho,\tau)])$ satisfying that $\xi=0$ on $\partial B_{\rho}(x_{0})\times(t_{0}-\tau,t_{0})$, $0\leq\xi\leq1$, it follows that for every $s\in(t_{0}-\tau,t_{0}),$
	\begin{align}\label{ENE01}
		&\int_{B_{\rho}(x_{0})}v_{\pm}^{2}\xi^{p}(x,s)dx+\frac{1}{6}\int_{B_{\rho}(x_{0})\times(t_{0}-\tau,s)}|\nabla(v_{\pm}\xi)|^{p}dxdt\notag\\
		&\leq\int_{B_{\rho}(x_{0})}v_{\pm}^{2}\xi^{p}(x,t_{0}-\tau)dx+C\int_{[(x_{0},t_{0})+Q(\rho,\tau)]}\big(v_{\pm}^{2}|\partial_{t}\xi|+v_{\pm}^{p}|\nabla\xi|^{p}\big)dxdt\notag\\
		&\quad+C\|f\|^{\frac{p}{p-1}}_{L^{\infty}(B_{1})}|[(x_{0},t_{0})+Q(\rho,\tau)]\cap \{v_{\pm}>0\}|.
	\end{align}
	
\end{lemma}

\begin{proof}
By translation, we may assume that $(x_{0},t_{0})=(0,0)$. In order to prove \eqref{ENE01}, we see from the proof of Proposition 3.1 in Chapter II of \cite{D1993} that it remains only to deal with the forcing term. First, using Poincar\'{e}'s inequality, we have
\begin{align*}
\int_{B_{\rho}\times(-\tau,s)}(u-k)_{\pm}^{p}\xi^{p}dxdt\leq C_{0}\int_{B_{\rho}\times(-\tau,s)}|\nabla((u-k)_{\pm}\xi)|^{p}dxdt,	
\end{align*}	
which, together with \eqref{AQ000} and Young's inequality, reads that for any $-\tau\leq s\leq0$,
	\begin{align}\label{DE90}
		&\int_{-\tau}^{s}\int_{B_{\rho}}\pm f(x)(u-k)_{\pm}\xi^{p}dxdt\notag\\
		&\leq \frac{1}{12C_{0}}\int_{B_{\rho}\times(-\tau,s)}(u-k)_{\pm}^{p}\xi^{p}dxdt+C\int_{(B_{\rho}\times(-\tau,s))\cap\{(u-k)_{\pm}>0\}}|f|^{\frac{p}{p-1}}\xi^{p}dxdt\notag\\
		&\leq\frac{1}{12}\int_{B_{\rho}\times(-\tau,s)}\big(|\nabla(u-k)_{\pm}|^{p}\xi^{p}+(u-k)_{\pm}^{p}|\nabla\xi|^{p}\big)dxdt\notag\\
		&\quad+C\|f\|^{\frac{p}{p-1}}_{L^{\infty}(B_{1})}|Q(\rho,\tau)\cap \{(u-k)_{\pm}>0\}|.
	\end{align}
Since $1<p<2$, we have  $(a+b)^{p}\leq2^{p-1}(a^{p}+b^{p})\leq 2(a^{p}+b^{p})$. Substituting this and \eqref{DE90} into the proof of Proposition 3.1 in Chapter II of \cite{D1993}, we deduce that \eqref{ENE01} holds.

\end{proof}

We are now ready to establish the global $L^{\infty}$ estimate as follows.
\begin{theorem}\label{CORO06}
Assume that $u$ is the weak solution of \eqref{PO001} with $\Omega_{T}=Q(1,1)$. Then we obtain
\begin{align}\label{DK01}
\|u\|_{L^{\infty}(Q(1,1))}\leq\|\phi\|_{L^{\infty}(B_{1})}+C_{1}\|f\|_{L^{\infty}(B_{1})}^{\frac{n+p}{(p-1)(n+2)}},
\end{align}
where $C_{1}=C_{1}(n,p)$.
\end{theorem}
\begin{proof}
For $i\geq0$, define $k_{i}=\|\phi\|_{L^{\infty}(B_{1})}+M-\frac{M}{2^{i}}$, where the constant $M>0$ will be specified later.\ Observe from \eqref{AQ000} that $(\pm u-k_{i})_{+}=0$ on $\partial_{pa}Q(1,1)$.\ Then taking $\xi\equiv1$ in \eqref{ENE01} and applying the proof of Lemma \ref{lem003} with minor modification, we have
\begin{align*}
&\sup\limits_{t\in(-1,0)}\int_{B_{1}}(u-k_{i})_{+}^{2}dx+\frac{1}{6}\int_{Q(1,1)}|\nabla(u-k_{i})_{+}|^{p}dxdt\notag\\
&\leq C\|f\|^{\frac{p}{p-1}}_{L^{\infty}(Q(1,1))}|Q(1,1)\cap\{u>k_{i}\}|,
\end{align*}
which, combined with Lemma \ref{lem001}, gives that
\begin{align*}
&\int_{Q(1,1)}|(u-k_{i})_{+}|^{p\tilde{\chi}}dxdt\notag\\
&\leq C\int_{Q(1,1)}|\nabla(u-k_{i})_{+}|^{p}dxdt\bigg(\sup\limits_{t\in(-1,0)}\int_{B_{1}}(u-k_{i})_{+}^{2}dx\bigg)^{\chi-1}\notag\\ 
&\leq C\|f\|^{\frac{p\chi}{p-1}}_{L^{\infty}(B_{1})}|Q(1,1)\cap\{u>k_{i}\}|^{\chi},
\end{align*}
where $\chi$ and $\tilde{\chi}$ are given in \eqref{chi}. To simplify the notation, write $A_{i}=Q(1,1)\cap\{u>k_{i}\}$. Then we have
\begin{align*}
&|A_{i+1}|\leq\Bigg(\frac{C\|f\|^{\frac{p}{p-1}}_{L^{\infty}(B_{1})}2^{\frac{p\tilde{\chi}(i+1)}{\chi}}}{M^{\frac{p\tilde{\chi}}{\chi}}}\Bigg)^{\chi}|A_{i}|^{\chi}\notag\\
&\leq\prod_{s=0}^{i}\Bigg(\frac{C\|f\|^{\frac{p}{p-1}}_{L^{\infty}(B_{1})}2^{\frac{p\tilde{\chi}(i+1)}{\chi}}}{M^{\frac{p\tilde{\chi}}{\chi}}}\Bigg)^{\chi^{s+1}}|A_{0}|^{\chi^{i+1}}\leq\Bigg[\bigg(\frac{\overline{C}_{0}\|f\|^{\frac{p}{p-1}}_{L^{\infty}(B_{1})}}{M^{\frac{p\tilde{\chi}}{\chi}}}\bigg)^{\frac{\chi}{\chi-1}}|A_{0}|\Bigg]^{\chi^{i+1}}.
\end{align*}
Pick
\begin{align*}
M=\Big(\overline{C}_{0}\|f\|^{\frac{p}{p-1}}_{L^{\infty}(B_{1})}\Big)^{\frac{\chi}{p\tilde{\chi}}}(2|Q(1,1)|)^{\frac{\chi-1}{p\tilde{\chi}}}.
\end{align*}
Hence we derive
\begin{align*}
|A_{i+1}|\leq2^{-\chi^{i+1}}\rightarrow0,\quad\text{as }i\rightarrow\infty,
\end{align*}
which implies that
\begin{align*}
\sup_{Q(1,1)}u\leq\|\phi\|_{L^{\infty}(B_{1})}+C\|f\|_{L^{\infty}(B_{1})}^{\frac{\chi}{\tilde{\chi}(p-1)}}.
\end{align*}
Applying the preceding argument to the equation satisfied by $-u$, we complete the proof.
\end{proof}

\section{Essentially time-uniform spatial distribution decay}\label{SEC005}

\subsection{Temporal Lipschitz regularity}
For every
\begin{align*} 
g\in C((-1,0];L^{p}(B_{1}))\cap L^{r}((-1,0);L^{q}(B_{1})),\quad \text{with }\, r,p,q>1, 
\end{align*}
the Steklov averages are defined as follows: for $0<h\ll1$,
\begin{align*}
	g_{h}(x,t)=
	\begin{cases}
		\frac{1}{h}\int^{t}_{t-h}g(x,s)ds,& t\in(-1+h,0),\\
		0,&t\leq-1+h,
	\end{cases}
\end{align*}
and
\begin{align*}
	g_{\bar{h}}(x,t)=
	\begin{cases}
		\frac{1}{h}\int^{t+h}_{t}g(x,s)ds,& t\in(-1,-h),\\
		0,&t\geq-h.
	\end{cases}
\end{align*}
Define
\begin{align}\label{DK02}
\mathcal{Q}=2\|\phi\|_{L^{\infty}(B_{1})}+\|f\|_{L^{\infty}(B_{1})}+C_{1}\|f\|_{L^{\infty}(B_{1})}^{\frac{n+p}{(p-1)(n+2)}},	
\end{align}	
where $C_{1}=C_{1}(n,p)$ is given by \eqref{DK01}.	The desired time-Lipschitz estimate is now stated as follows.
\begin{theorem}\label{thm10}
Let $1<p<2$. Assume that $u$ is the solution to problem \eqref{PO001} with $\Omega\times(-T,0]=B_{1}\times(-1,0]$. Then we obtain that for every $-1<s<t\leq 0$,
	\begin{align}\label{AH01}
		\|u(\cdot,t)-u(\cdot,s)\|_{L^{\infty}(B_{1})}
		\leq\frac{\mathcal{Q}}{2-p}\frac{t-s}{t+1}.
	\end{align}
In particular, 
	\begin{align}\label{AH03}
		|\partial_{t}u(x,t)|\leq \frac{4\mathcal{Q}}{2-p},\qquad\text{for a.e. }(x,t)\in B_{1}\times\Big(-\frac{3}{4},0\Big],
	\end{align}
where $\mathcal{Q}$ is defined by \eqref{DK02}.
\end{theorem}
\begin{remark}
The technique used in the proof of Theorem \ref{thm10} can be summarized as a combination of temporal rescaling and an $L^\infty$ comparison estimate.\ In addition, the proof also shows that the derivation of the time-Lipschitz estimate relies on the assumption that both the forcing term $f$ and the Dirichlet data $\phi$ are independent of time.

\end{remark}
\begin{proof}
For any $-1<s<t\leq0,$ introduce a scaling factor $\lambda:=(\frac{s+1}{t+1})^{\frac{1}{2-p}}\in(0,1)$. For $\tau\in(-1,s]$, define 
\begin{align*}
\mathcal{M}_{\tau}=(1-\lambda)\|\phi\|_{L^{\infty}(B_{1})}+(1+\tau)(1-\lambda^{p-1})\|f\|_{L^{\infty}(B_{1})}.
\end{align*}
Denote
\begin{align*}
\tilde{u}(x,\tau)=\lambda u(x,-1+\lambda^{p-2}(\tau+1)).	
\end{align*}	
Observe that $\tilde{u}$ satisfies 
\begin{align*}
	\begin{cases}
		\partial_{\tau}\tilde{u}-\mathrm{div}(|\nabla \tilde{u}|^{p-2}\nabla \tilde{u})=\lambda^{p-1}f(x),& \mathrm{in}\;B_{1}\times(-1,s],\\
		\tilde{u}=\lambda\phi(x),&\mathrm{on}\;\partial_{pa}(B_{1}\times(-1,s]).
	\end{cases}
\end{align*}
Then $v(x,\tau):=u(x,\tau)-\tilde{u}(x,\tau)-\mathcal{M}_{\tau}\in C([-1,s];L^{2}(B_{1}))\cap L^{p}((-1,s);W^{1,p}(B_{1}))$ satisfies 
 \begin{align}\label{AH06}
 		&\partial_{\tau}v-\mathrm{div}(|\nabla u|^{p-2}\nabla u-|\nabla \tilde{u}|^{p-2}\nabla \tilde{u})\notag\\
 		&=(1-\lambda^{p-1})(f(x)-\|f\|_{L^{\infty}(B_{1})})\leq0,\quad \mathrm{in}\;B_{1}\times(-1,s],
 \end{align}	
with the boundary condition of
\begin{align}\label{AH10}
v=&(1-\lambda)(\phi(x)-\|\phi\|_{L^{\infty}(B_{1})})-(1+\tau)(1-\lambda^{p-1})\|f\|_{L^{\infty}(B_{1})}\notag\\
\leq&0,\quad\mathrm{on}\;\partial_{pa}(B_{1}\times(-1,s]).	
\end{align}		
For $0<h\ll 1$, denote $\varphi=(v_{h})_{+}=\max\{v_{h},0\}$, where $v_{h}$ is the Steklov average of $v$.\ Then multiplying \eqref{AH06} by the test function $\varphi_{\bar{h}}\in W^{1,2}((-1,s);L^{2}(B_{1}))\cap L^{p}((-1,s);W^{1,p}_{0}(B_{1}))$, we have
\begin{align*}
	&\int_{B_{1}}v\varphi_{\bar{h}}dx\Big|_{-1}^{s}+\int_{-1}^{s}\int_{B_{1}}\big(-v\partial_{\tau}\varphi_{\bar{h}}+(|\nabla u|^{p-2}\nabla u-|\nabla \tilde{u}|^{p-2}\nabla \tilde{u})\cdot\nabla\varphi_{\bar{h}}\big)dxd\tau\notag\\
	&=\int_{-1}^{s}\int_{B_{1}}(1-\lambda^{p-1})(f(x)-\|f\|_{L^{\infty}(B_{1})})\varphi_{\bar{h}}dxd\tau.
\end{align*}	
Then we obtain 	
\begin{align}\label{AH08}
	&\int_{-1}^{s}\int_{B_{1}}(\partial_{\tau}v_{h}\varphi+(|\nabla u|^{p-2}\nabla u-|\nabla \tilde{u}|^{p-2}\nabla \tilde{u})_{h}\cdot\nabla\varphi)dxd\tau\notag\\
	&=\int_{-1}^{s}\int_{B_{1}}(1-\lambda^{p-1})(f(x)-\|f\|_{L^{\infty}(B_{1})})\varphi dxd\tau\leq0.
\end{align}	
It then follows from integration by parts and Lemma 3.2 in Chapter I of \cite{D1993} that
\begin{align*}
&\int_{-1}^{s}\int_{B_{1}}\partial_{\tau}v_{h}\varphi dxd\tau=\frac{1}{2}\int_{-1}^{s}\int_{B_{1}}\partial_{\tau}[(v_{h})_{+}]^{2} dxd\tau\notag\\
&=\frac{1}{2}\int_{B_{1}}[(v_{h})_{+}]^{2}(x,s) dx-\frac{1}{2}\int_{B_{1}}[(v_{h})_{+}]^{2}(x,-1) dx\notag\\
&\rightarrow\frac{1}{2}\int_{B_{1}}(v_{+})^{2}(x,s) dx-\frac{1}{2}\int_{B_{1}}(v_{+})^{2}(x,-1) dx,\quad\text{as }h\rightarrow0,
\end{align*}	
and
\begin{align*}
&\int_{-1}^{s}\int_{B_{1}}(|\nabla u|^{p-2}\nabla u-|\nabla \tilde{u}|^{p-2}\nabla \tilde{u})_{h}\cdot\nabla\varphi dxd\tau\notag\\
&\rightarrow\int_{(B_{1}\times(-1,s))\cap \{v_{+}>0\}}(|\nabla u|^{p-2}\nabla u-|\nabla \tilde{u}|^{p-2}\nabla \tilde{u})\cdot\nabla (u-\tilde{u})dxd\tau,\quad\text{as }h\rightarrow0.
\end{align*}
Observe from Lemma \ref{Lem0006} that
\begin{align*}
\int_{(B_{1}\times(-1,s))\cap \{v_{+}>0\}}(|\nabla u|^{p-2}\nabla u-|\nabla \tilde{u}|^{p-2}\nabla \tilde{u})\cdot\nabla (u-\tilde{u})dxd\tau\geq0.	
\end{align*}	
Therefore, substituting these above terms into \eqref{AH08}, we deduce from \eqref{AH10} that
\begin{align*}
&\int_{B_{1}}(v_{+})^{2}(x,s) dx\leq \int_{B_{1}}(v_{+})^{2}(x,-1) dx=0,
\end{align*}	
which leads to that 
\begin{align}\label{AH13}
u(x,s)-\tilde{u}(x,s)\leq \mathcal{M}_{\tau},\quad\text{for a.e.}\;x\in B_{1}.	
\end{align}	

On the other hand, we see that $w:=\tilde{u}-u$ is the solution of
 \begin{align*}
	&\partial_{\tau}w-\mathrm{div}(|\nabla \tilde{u}|^{p-2}\nabla \tilde{u}-|\nabla u|^{p-2}\nabla u)\notag\\
	&=(\lambda^{p-1}-1)f(x)\leq0,\quad \mathrm{in}\;B_{1}\times(-1,s],
\end{align*}	
satisfying the boundary condition of
\begin{align*}
	w=&(\lambda-1)\phi(x)\leq0,\quad\mathrm{on}\;\partial_{pa}(B_{1}\times(-1,s]).	
\end{align*}		
Then repeating the above argument, we obtain 
\begin{align*}
\tilde{u}(x,s)-u(x,s)\leq 0,\quad\text{for a.e.}\;x\in B_{1}.
\end{align*}
This, together with \eqref{AH13}, gives that
\begin{align*}
\|u(\cdot,s)-\tilde{u}(\cdot,s)\|_{L^{\infty}(B_{1})}\leq \mathcal{M}_{\tau}\leq(1-\lambda)(\|\phi\|_{L^{\infty}(\Omega)}+\|f\|_{L^{\infty}(\Omega)}),
\end{align*}	
where we also used the fact of $1<p<2$ and $\lambda\in(0,1)$. From the definition of $\lambda$, we have
\begin{align*}
\tilde{u}(x,s)=\lambda u(x,t),\quad \text{for }x\in B_{1}\text{ and }-1<s<t\leq0.	
\end{align*}	
Consequently, we deduce from Theorem \ref{CORO06} that for $-1<s<t\leq0$,
\begin{align*}
&\|u(\cdot,t)-u(\cdot,s)\|_{L^{\infty}(B_{1})}\leq(1-\lambda)\|u(\cdot,t)\|_{L^{\infty}(B_{1})}+\|\lambda u(\cdot,t)-u(\cdot,s)\|_{L^{\infty}(B_{1})}\notag\\
&\leq(1-\lambda)\|u(\cdot,t)\|_{L^{\infty}(B_{1})}+\| \tilde{u}(\cdot,s)-u(\cdot,s)\|_{L^{\infty}(B_{1})}\notag\\
&\leq\mathcal{Q}\bigg(1-\Big(\frac{s+1}{t+1}\Big)^{\frac{1}{2-p}}\bigg)\leq\frac{\mathcal{Q}}{2-p}\frac{t-s}{t+1},
\end{align*}	
where $\mathcal{Q}$ is defined by \eqref{DK02}, and we also utilized the following inequality:
\begin{align*}
a^{q}\geq 1+q(a-1),\quad \text{for any }a\geq0,\;q\geq1.	
\end{align*}

\end{proof}		
	
Using Theorem \ref{thm10}, we obtain a time-slice Caccioppoli estimate as follows.		
\begin{prop}\label{prop005}
Suppose that $u$ is the solution of \eqref{PO001} with $\Omega\times(-T,0]=B_{1}\times(-1,0]$. Let $k\in\mathbb{R}$ and $B_{\rho}(x_{0})\subset B_{1}$. For any time-independent cutoff function $\zeta\in C_{0}^{\infty}(B_{\rho}(x_{0}))$ with $0\leq\zeta\leq1$, we obtain that for almost every $t\in (-\frac{3}{4},0]$,
	\begin{align*}
		&\int_{B_{\rho}(x_{0})}|\nabla((u-k)_{\pm}(\cdot,t)\zeta)|^{p}dx\notag\\
		&\leq C_{2}\int_{B_{\rho}(x_{0})}|(u-k)_{\pm}(\cdot,t)\nabla\zeta|^{p}dx
		+C_{2}\mathcal{Q}\int_{B_{\rho}(x_{0})}(u-k)_{\pm}(\cdot,t)\zeta^{p}dx,
	\end{align*}
where $C_{2}=C_{2}(p)$ and $\mathcal{Q}$ is given by \eqref{DK02}.	
\end{prop}
\begin{remark}
Based on the energy inequality in Proposition \ref{prop005} and using the classical elliptic De Giorgi truncation method in \cite{D1957}, we directly obtain that there exists some $\alpha\in(0,1)$ such that for any $x,y\in B_{\rho}$ and $\rho<1$,
\begin{align*}
|u(x,t)-u(y,t)|\leq C|x-y|^{\alpha},\quad\text{for a.e. $t\in\Big(-\frac{3}{4},0\Big]$. }	
\end{align*}	 
However, in order to obtain H\"{o}lder continuity throughout a local space-time region, we need to return to the proof framework of the parabolic De Giorgi truncation method and establish the corresponding estimates therein.	
\end{remark}	

\begin{proof}
From Theorem \ref{thm10}, we see that $\partial_{t}u\in L^{\infty}$. Then picking the test function $\varphi=\pm (u-k)_{\pm}\zeta^{p}$, we have
\begin{align*}
\int_{B_{\rho}(x_{0})}|\nabla u|^{p-2}\nabla u\cdot\nabla\varphi dx=\int_{B_{\rho}(x_{0})}(f-\partial_{t}u)\varphi dx.		
\end{align*}
On one hand, it follows from Young's inequality that
\begin{align}\label{AH18}
&\int_{B_{\rho}(x_{0})}|\nabla u|^{p-2}\nabla u\cdot\nabla\varphi dx\notag\\
&=\int_{B_{\rho}(x_{0})}|\nabla u|^{p-2}\nabla u\cdot(\pm\zeta^{p}\nabla (u-k)_{\pm}\pm p\zeta^{p-1}(u-k)_{\pm}\nabla\zeta)dx\notag\\	
&\geq\int_{B_{\rho}(x_{0})}|\zeta\nabla (u-k)_{\pm}|^{p}dx-p\int_{B_{\rho}(x_{0})}|\zeta\nabla (u-k)_{\pm}|^{p-1}|(u-k)_{\pm}\nabla\zeta|dx\notag\\
&\geq \frac{1}{2}\int_{B_{\rho}(x_{0})}|\zeta\nabla (u-k)_{\pm}|^{p}dx-C(p)\int_{B_{\rho}(x_{0})}|(u-k)_{\pm}\nabla\zeta|^{p}dx.
\end{align}	

On the other hand, we have from \eqref{AH03} that for almost every $t\in(-\frac{3}{4},0]$,
\begin{align}\label{AH19}
\int_{B_{\rho}(x_{0})}(f-\partial_{t}u)\varphi dx\leq& \Big(\|f\|_{L^{\infty}(\Omega)}+\frac{4\mathcal{Q}}{2-p}\Big)\int_{B_{\rho}(x_{0})}(u-k)_{\pm}\zeta^{p}dx\notag\\
\leq&\frac{5\mathcal{Q}}{2-p}\int_{B_{\rho}(x_{0})}(u-k)_{\pm}\zeta^{p}dx.
\end{align}	
Therefore, combining the inequality of $(a+b)^{p}\leq2^{p-1}(a^{p}+b^{p})$, we deduce from \eqref{AH18}--\eqref{AH19} that Proposition \ref{prop005} holds.

\end{proof}

\subsection{Essentially uniform-in-time spatial distribution decay estimate}
For $1<p<2$, denote
\begin{align*}
	\varepsilon_{0}=\frac{2-p}{2}.
\end{align*}
Then we derive that for every $R\in(0,2^{-2/p}]$, $x_{0}\in \overline{B}_{\frac{1}{2}}$, and $t_{0}\in[-\frac{1}{2},0]$,
\begin{align*}
	[(x_{0},t_{0})+Q(R^{1-\varepsilon_{0}},R^{p})]\subset B_{1}\times\Big(-\frac{3}{4},0\Big]\subset Q(1,1).
\end{align*}
It is worth remarking that since $R\in(0,2^{-2/p}]$ and $x_{0}\in \overline{B}_{\frac{1}{2}}$, we have
\begin{align}\label{DOM01}
	2^{-1}+R^{1-\varepsilon_{0}}=2^{-1}+R^{\frac{p}{2}}\leq 1,
\end{align}
which ensures that the cylinder under consideration is contained in $Q(1,1)$. This actually restricts the region in which we derive the spatial regularity below to $B_{2^{-1}+2^{-2/p}}$.\ By translating the spatial and temporal variables, we may assume without loss of generality that $(x_{0},t_{0})=(0,0)$. Write
\begin{align*}
	%\label{W08}
	\mu^{+}=\sup\limits_{Q(R^{1-\varepsilon_{0}},R^{p})}u,\quad\mu^{-}=\inf\limits_{Q(R^{1-\varepsilon_{0}},R^{p})}u,
\end{align*}
and
\begin{align*}
	%\label{W09}
	\omega=\mathop{osc}\limits_{Q(R^{1-\varepsilon_{0}},R^{p})}u=\mu^{+}-\mu^{-}.
\end{align*}
Introduce the following cylinder:
\begin{align}\label{Z999}
	Q(a_{0}R,R^{p}),\quad a_{0}:=\Big(\frac{\sigma\omega}{A}\Big)^{\frac{p-2}{p}},
\end{align}
where $\sigma\in(0,1)$ and $A\in(1,\infty)$ are constants whose values will be specified later and depend only on the prescribed data.\ It is worth noting that $a_{0}$ is referred to as the intrinsic scaling factor, the rescaled cylinders $Q(a_{0}R,R^{p})$ are adapted to the singular structure of the equation.\ The exponent $\frac{p-2}{p}$ arising in this scaling was already identified in the earlier works \cite{CD1988,D1993}.

We distinguish between the following two cases: $\omega\leq \sigma^{-1}AR^{\frac{p}{2}}$, or $\omega>\sigma^{-1}AR^{\frac{p}{2}}$. When the latter holds, we have
\begin{align*}
	Q(a_{0}R,R^{p})\subset Q(R^{1-\varepsilon_{0}},R^{p}),\quad \mathop{osc}\limits_{Q(a_{0}R,R^{p})}u\leq\omega.
\end{align*}
For brevity, write
\begin{align}\label{M01}
	m=\Big(\frac{\sigma\omega}{\mathcal{M}_{\ast}}\Big)^{\frac{p-2}{p}},\quad\mathcal{M}_{\ast}=2\Big(\|\phi\|_{L^{\infty}(B_{1})}+C_{1}\|f\|_{L^{\infty}(B_{1})}^{\frac{n+p}{(p-1)(n+2)}}\Big)+1,
\end{align}
where $C_{1}$ is given by \eqref{DK01}. Hence we obtain that $m\geq1$, and
\begin{align}\label{M02}
	\mathcal{M}_{\ast}\geq2\|u\|_{L^{\infty}(Q(1,1))}+1.
\end{align}	
Under the condition of $A\geq\mathcal{M}_{\ast}$, we derive $Q(mR,R^{p})\subset Q(a_{0}R,R^{p})$. 

 We now state the following desired essentially uniform-in-time spatial level-set measure decay estimate.
\begin{prop}\label{prop009}
Assume that $\sigma,m,\mathcal{M}_{\ast}$ are defined by \eqref{Z999}--\eqref{M01}.\ There exists a constant $\widehat{C}_{0}>1$ depending only on the data but independent of $\sigma$ such that for any $\gamma\in(0,1)$ and $\delta\in(0,\frac{1}{2}]$,
	
	$(i)$ if
	\begin{align}\label{AH20}
		\frac{| B_{mR/2}\cap\{u(\cdot,t)>\mu^{+}-\delta\omega\}|}{|B_{mR/2}|}\leq1-\gamma,\quad\forall\;t\in[-R^{p},0],
	\end{align}
	then  for any $j\geq1$, one has either
	\begin{align*}
		\omega\leq\sigma^{-1}\mathcal{M}_{\ast}\max\Big\{1,\Big(\frac{2^{j}\sigma}{\delta\mathcal{M}_{\ast}}\Big)^{p-1}\Big\}R^{\frac{p}{2}},
	\end{align*}
	or
	\begin{align}\label{AH21}
			\operatorname*{ess\,sup}_{t\in(-R^{p},0)}\frac{| B_{mR/2}\cap\{u(\cdot,t)>\mu^{+}-\frac{\delta\omega}{2^{j}}\}|}{|B_{mR/2}|}\leq\frac{\widehat{C}_{0}}{\gamma j^{\frac{p-1}{p}}};
	\end{align}
	
	$(ii)$ if
	\begin{align}\label{AH22}
		\frac{|B_{mR/2}\cap\{u(\cdot,t)<\mu^{-}+\delta\omega\}|}{|B_{mR/2}|}\leq1-\gamma,\quad\forall\;t\in[-R^{p},0],
	\end{align}
	then for any $j\geq1$, one has either
	\begin{align*}
		\omega\leq\sigma^{-1}\mathcal{M}_{\ast}\max\Big\{1,\Big(\frac{2^{j}\sigma}{\delta\mathcal{M}_{\ast}}\Big)^{p-1}\Big\}R^{\frac{p}{2}},
	\end{align*}
	or
	\begin{align}\label{AH23}
	\operatorname*{ess\,sup}_{t\in(-R^{p},0)}\frac{| B_{mR/2}\cap\{u(\cdot,t)<\mu^{-}+\frac{\delta\omega}{2^{j}}\}|}{|B_{mR/2}|}\leq\frac{\widehat{C}_{0}}{\gamma j^{\frac{p-1}{p}}}.
	\end{align}
\end{prop}

\begin{proof}
	\noindent{\bf Part 1.}
	For $i\geq0$, denote $k_{i}=\mu^{+}-\frac{\delta\omega}{2^{i}}$, and
	\begin{align*}
		A_{i}(t)=B_{mR/2}\cap\{u(\cdot,t)>k_{i}\},\quad \text{for a.e. }t\in(-R^{p},0].
	\end{align*}
Utilizing \eqref{AH20}, we have
	\begin{align*}
		|B_{mR/2}\setminus A_{i}(t)|\geq\gamma|B_{mR/2}|\geq C(n)\gamma(mR)^{n}.
	\end{align*}
This, in combination with \eqref{pro001}, reads that
	\begin{align*}
		|A_{i+1}(t)|\leq\frac{C2^{i}mR}{\gamma\delta\omega}\int_{A_{i}(t)\setminus A_{i+1}(t)}|\nabla u|dx.
	\end{align*}
In view of $1<p<2$, we have from H\"{o}lder's inequality that
	\begin{align*}
	\int_{A_{i}(t)\setminus A_{i+1}(t)}|\nabla u|dx\leq&\bigg(\int_{A_{i}(t)\setminus A_{i+1}(t)}|\nabla u|^{p}dx\bigg)^{\frac{1}{p}}|A_{i}(t)\setminus A_{i+1}(t)|^{\frac{p-1}{p}}\notag\\
		\leq&\bigg(\int_{mR/2}|\nabla (u-k_{i})_{+}|^{p}dx\bigg)^{\frac{1}{p}}|A_{i}(t)\setminus A_{i+1}(t)|^{\frac{p-1}{p}}.
	\end{align*}
Choose a time-independent cutoff function $\zeta\in C_{0}^{\infty}(B_{mR})$ such that
	\begin{align*}
		\zeta=1\;\mathrm{in}\;B_{mR/2},\quad\mathrm{and}\;\,0\leq\zeta\leq1,\;|\nabla\zeta|\leq\frac{4}{mR}\;\,\mathrm{in}\;B_{mR}.
	\end{align*}
A combination of Proposition \ref{prop005} and \eqref{M01} leads to that if
	\begin{align*}
		\omega>&\sigma^{-1}\mathcal{M}_{\ast}\max\Big\{1,\Big(\frac{2^{i}\sigma}{\delta\mathcal{M}_{\ast}}\Big)^{p-1}\Big\}R^{\frac{p}{2}}\notag\\
		\geq&\max\Big\{\sigma^{-1}\mathcal{M}_{\ast},(2^{i}\delta^{-1})^{p-1}(\sigma^{-1}\mathcal{M}_{\ast})^{2-p}R^{\frac{p}{2}}\Big\}R^{\frac{p}{2}},
	\end{align*}
	then
	\begin{align*}
		&\int_{B_{mR/2}}|\nabla (u-k_{i})_{+}|^{p}dx\leq\int_{B_{mR}}|\nabla((u-k_{i})_{+}\zeta)|^{p}dx\notag\\
		&\leq C\int_{B_{mR}}\big(|(u-k_{i})_{+}\nabla\zeta|^{p}+(u-k_{i})_{+}\zeta^{p}\big)dx\notag\\
		&\leq C\Big(\frac{(\delta\omega)^{p}(mR)^{n-p}}{2^{pi}}+\frac{\delta\omega(mR)^{n}}{2^{i}}\Big)\leq \frac{C(\delta\omega)^{p}(mR)^{n-p}}{2^{pi}}.
	\end{align*}
Combining these above facts, we obtain
	\begin{align*}
		|A_{i+1}(t)|\leq&\frac{C}{\gamma}|A_{i}(t)\setminus A_{i+1}(t)|^{\frac{p-1}{p}}|B_{mR/2}|^{\frac{1}{p}}.
	\end{align*}
Hence we deduce that for any $j>i\geq0$,  if
\begin{align*}
	\omega>&\sigma^{-1}\mathcal{M}_{\ast}\max\Big\{1,\Big(\frac{2^{j}\sigma}{\delta\mathcal{M}_{\ast}}\Big)^{p-1}\Big\}R^{\frac{p}{2}},
\end{align*}
we have
	\begin{align*}
		j|A_{j}(t)|^{\frac{p}{p-1}}\leq&\sum^{j-1}_{i=0}|A_{i+1}(t)|^{\frac{p}{p-1}}\leq\frac{C}{\gamma^{\frac{p}{p-1}}}|B_{mR/2}|^{\frac{p}{p-1}},\quad\text{for a.e. }t\in(-R^{p},0],
	\end{align*}
which shows that \eqref{AH21} holds.
	
\noindent{\bf Part 2.}
	Analogously as above, for $i\geq0$, write $\tilde{k}_{i}=\mu^{-}+\frac{\delta\omega}{2^{i}}$ and
	\begin{align*}
		\tilde{A}_{i}(t)=B_{mR/2}\cap\{u(\cdot,t)<k_{i}\},\quad \text{for a.e. }t\in(-R^{p},0].
	\end{align*}
Making use of \eqref{AH22}, we deduce
	\begin{align*}
		|B_{mR/2}\setminus \tilde{A}_{i}(t)|\geq\gamma|B_{mR/2}|\geq C(n)\gamma(mR)^{n}.
	\end{align*}
Using \eqref{pro001}, we have
	\begin{align*}
		|\tilde{A}_{i+1}(t)|\leq\frac{C2^{i}mR}{\gamma\delta\omega}\int_{\tilde{A}_{i}(t)\setminus \tilde{A}_{i+1}(t)}|\nabla u|dx.
	\end{align*}
Then by the same argument as in the proof of \eqref{AH21}, we derive that \eqref{AH23} holds.

\end{proof}

A direct consequence of Proposition \ref{prop009} gives the following space-time decay estimate for the measure of level sets.

\begin{corollary}\label{coro008}
Suppose that $\sigma,m,\mathcal{M}_{\ast}$ are given by \eqref{Z999}--\eqref{M01}.\ There exists a $\sigma$-independent constant $\widehat{C}_{0}>1$ depending only upon the data such that for any $\gamma\in(0,1)$ and $\delta\in(0,\frac{1}{2}]$, if
	\begin{align*}
		\frac{| B_{mR/2}\cap\{(\mu^{\pm}-u(\cdot,t))_{\pm}<\delta\omega\}|}{|B_{mR/2}|}\leq1-\gamma,\quad\forall\;t\in[-R^{p},0],
	\end{align*}
	then for every $j\geq1$, one has either
	\begin{align*}
		\omega\leq\sigma^{-1}\mathcal{M}_{\ast}\max\Big\{1,\Big(\frac{2^{j}\sigma}{\delta\mathcal{M}_{\ast}}\Big)^{p-1}\Big\}R^{\frac{p}{2}},
	\end{align*}
	or
	\begin{align*}
	\frac{| Q(mR/2,(R/2)^{p})\cap\{(\mu^{\pm}-u)_{\pm}<\frac{\delta\omega}{2^{j}}\}|}{|Q(mR/2,(R/2)^{p})|}\leq\frac{\widehat{C}_{0}}{\gamma j^{\frac{p-1}{p}}}.
	\end{align*}

\end{corollary}

\section{H\"{o}lder estimates for weak solutions }\label{SEC006}
In view of the space-time level-set distribution decay estimate obtained in Corollary \ref{coro008}, we next establish a measure-to-point oscillation improvement estimate and a time-expansion lemma.\ These three ingredients together form the core procedure of the parabolic De Giorgi truncation method.

\begin{lemma}\label{LEM0035}
Suppose that $\sigma,m,\mathcal{M}_{\ast}$ are given by \eqref{Z999}--\eqref{M01}. Then there exists a constant $\gamma_{0}\in(0,1)$ depending only on the prescribed data but independent of $\sigma$ such that

$(i)$ if
\begin{align}\label{ZWZ007}
|Q(mR,R^{p})\cap\{u>\mu^{+}-\sigma\omega\}|\leq\gamma_{0}|Q(mR,R^{p})|,
\end{align}
then one has either $\omega\leq \sigma^{-1}\mathcal{M}_{\ast}R^{\frac{p}{2}}$, or
\begin{align}\label{DZ001}
u\leq \mu^{+}-\frac{\sigma\omega}{2},\quad\mathrm{in}\text{ $Q(mR/2,(R/2)^{p})$;}
\end{align}

$(ii)$ if
\begin{align*}
|Q(mR,R^{p})\cap\{u<\mu^{-}+\sigma\omega\}|\leq\gamma_{0}|Q(mR,R^{p})|,
\end{align*}
then one has either $\omega\leq \sigma^{-1}\mathcal{M}_{\ast}R^{\frac{p}{2}}$, or
\begin{align}\label{ZWZ009}
u\geq \mu^{-}+\frac{\sigma\omega}{2},\quad\mathrm{in}\text{ $Q(mR/2,(R/2)^{p})$.}
\end{align}

\end{lemma}

\begin{proof}
It suffices to prove \eqref{DZ001}, as the proof of \eqref{ZWZ009} is analogous. For $i\geq0$, denote
\begin{align*}
r_{i}=\frac{R}{2}+\frac{R}{2^{i+1}},\quad k_{i}=\mu^{+}-\sigma\omega+\frac{\sigma\omega}{2}(1-2^{-i}).
\end{align*}
Let $\xi_{i}\in C^{\infty}(Q(mr_{i},r_{i}^{p}))$ be a smooth cutoff function satisfying that
\begin{align*}
\begin{cases}
0\leq\xi_{i}\leq1,&\mathrm{in}\;Q(mr_{i},r_{i}^{p}),\\
\xi_{i}=1,&\mathrm{in}\;Q(mr_{i+1},r_{i+1}^{p}),\\
\xi_{i}=0,&\mathrm{on}\;\partial_{pa}Q(mr_{i},r_{i}^{p}),\\
|\nabla\xi_{i}|\leq\frac{2^{i+3}}{mR},\quad|\partial_{t}\xi_{i}|\leq\frac{2}{r_{i}^{p}-r_{i+1}^{p}}\leq\frac{C2^{i}}{R^{p}}.
\end{cases}
\end{align*}
In light of \eqref{M01}, we have
\begin{align}\label{Z90}
(\sigma\omega)^{2}=\Big(\frac{\sigma\omega}{m}\Big)^{p}\mathcal{M}_{\ast}^{2-p}.
\end{align}
Note that $\xi_{i}^{2}\leq\xi_{i}^{p}$ in $Q(mr_{i},r_{i}^{p})$ in virtue of $1<p<2$. Then combining \eqref{ENE01} and \eqref{Z90}, we deduce that if $\omega>\sigma^{-1}\mathcal{M}_{\ast}R^{\frac{p}{2}}$,
\begin{align*}
&\sup\limits_{t\in(-r_{i}^{p},0)}\int_{B_{mr_{i}}}|(u-k_{i})_{+}\xi_{i}|^{2}dx+\frac{1}{6}\int_{Q(mr_{i},r_{i}^{p})}|\nabla((u-k_{i})_{+}\xi_{i})|^{p}dxdt\notag\\
&\leq\sup\limits_{t\in(-r_{i}^{p},0)}\int_{B_{mr_{i}}}(u-k_{i})^{2}_{+}\xi_{i}^{p}dx+\frac{1}{6}\int_{Q(mr_{i},r_{i}^{p})}|\nabla((u-k_{i})_{+}\xi_{i})|^{p}dxdt\notag\\
&\leq C\int_{Q(mr_{i},r_{i}^{p})}\big((u-k_{i})_{+}^{2}|\partial_{t}\xi_{i}|+(u-k_{i})_{+}^{p}|\nabla\xi_{i}|^{p}\big)dxdt+C|Q(mr_{i},r_{i}^{p})\cap \{u>k_{i}\}|\notag\\
&\leq C\bigg(\frac{2^{i}(\sigma\omega)^{2}}{R^{p}}+2^{pi}\Big(\frac{\sigma\omega}{mR}\Big)^{p}+1\bigg)|Q(mr_{i},r_{i}^{p})\cap \{u>k_{i}\}|\notag\\
&\leq C\bigg(\frac{2^{pi}(\sigma\omega)^{2}}{R^{p}}+1\bigg)|Q(mr_{i},r_{i}^{p})\cap \{u>k_{i}\}|\leq\frac{C2^{pi}(\sigma\omega)^{2}}{R^{p}}|Q(mr_{i},r_{i}^{p})\cap \{u>k_{i}\}|.
\end{align*}
It then follows from Lemma \ref{lem001} that
\begin{align}\label{ZE01}
&\Big(\frac{\sigma\omega}{2^{i+2}}\Big)^{p\tilde{\chi}}|Q(mr_{i+1},r_{i+1}^{p})\cap \{u>k_{i+1}\}|\notag\\
&=(k_{i+1}-k_{i})^{p\tilde{\chi}}|Q(mr_{i+1},r_{i+1}^{p})\cap \{u>k_{i+1}\}|\notag\\
&\leq\int_{Q(mr_{i},r_{i}^{p})}|(u-k_{i})_{+}\xi_{i}|^{p\tilde{\chi}}dxdt\notag\\
&\leq C\bigg(\sup\limits_{t\in(-r_{i}^{p},0)}\int_{B_{mr_{i}}}|(u-k_{i})_{+}\xi_{i}|^{2}dx\bigg)^{\chi-1}\int_{Q(mr_{i},r_{i}^{p})}|\nabla((u-k_{i})_{+}\xi_{i})|^{p}dxdt\notag\\
&\leq\bigg(\frac{C2^{pi}(\sigma\omega)^{2}}{R^{p}}|Q(mr_{i},r_{i}^{p})\cap \{u>k_{i}\}|\bigg)^{\chi}.
\end{align}
Write
\begin{align*}
F_{i}=\frac{|Q(mr_{i},r_{i}^{p})\cap \{u>k_{i}\}|}{|Q(mR,R^{p})|},\quad\mathrm{for}\;i\geq0.
\end{align*}
Using \eqref{chi} and \eqref{M01}, we derive
\begin{align*}
(\sigma\omega)^{2-\frac{p\tilde{\chi}}{\chi}}=m^{-\frac{np}{n+p}}\mathcal{M}_{\ast}^{\frac{n(2-p)}{n+p}},
\end{align*}
which, in conjunction with \eqref{ZE01}, implies that
\begin{align*}
F_{i+1}\leq&\Big(C2^{pi(1+\frac{\tilde{\chi}}{\chi})}F_{i}\Big)^{\chi}\leq\prod\limits^{i}_{s=0}\Big(C2^{p(1+\frac{\tilde{\chi}}{\chi})(i-s)}\Big)^{\chi^{s+1}}F_{0}^{\chi^{i+1}}\notag\\
\leq&C^{\sum\limits^{i}_{s=0}\chi^{s+1}}2^{p(\chi+\tilde{\chi})\sum\limits_{s=0}^{i}(i-s)\chi^{s}}F_{0}^{\chi^{i+1}}\leq(\widetilde{C}_{0}F_{0})^{\chi^{i+1}},
\end{align*}
where $\widetilde{C}_{0}$ is a $\sigma$-independent positive constant. Then taking $\gamma_{0}=(2\widetilde{C}_{0})^{-1}$ in \eqref{ZWZ007}, we deduce 
\begin{align*}
F_{i+1}\leq2^{-\chi^{i+1}}\rightarrow0,\quad\mathrm{as}\;i\rightarrow\infty.
\end{align*}
The proof is complete.

\end{proof}

We next establish a forward-in-time propagation estimate for the distribution function of the solution. Introduce a constant as follows:
\begin{align}\label{AH26}
\kappa_{0}:=-\frac{\ln(1-2/\sqrt{5})}{\ln 2}>0.	
\end{align}	
\begin{lemma}\label{LEM006}
Let $\sigma,m,\mathcal{M}_{\ast}$ be defined by \eqref{Z999}--\eqref{M01}. There exists a small constant $0<\sigma_{0}<1$, depending only on the prescribed data but not on $\sigma$, such that for every $0<\sigma\leq\sigma_{0}$,

$(i)$ if
\begin{align}\label{WMDN001}
\frac{|B_{mR}\cap\{u(\cdot,-R^{p})>\mu^{+}-2^{-1}\omega\}|}{|B_{mR}|}\leq\frac{1}{2},
\end{align}
then one has either $\omega\leq \sigma^{-1}\mathcal{M}_{\ast}R^{\frac{p}{2}}$, or
\begin{align}\label{DQAF001}
\frac{|B_{mR}\cap\{u(\cdot,t)>\mu^{+}-2^{-(\kappa_{0}+1)}\omega\}|}{|B_{mR}|}\leq\frac{3}{4},\quad\text{for any }t\in[-R^{p},0];
\end{align}

$(ii)$ if
\begin{align*}
\frac{|B_{mR}\cap\{u(\cdot,-R^{p})<\mu^{-}+2^{-1}\omega\}|}{|B_{mR}|}\leq\frac{1}{2},
\end{align*}
then one has either $\omega\leq \sigma^{-1}\mathcal{M}_{\ast}R^{\frac{p}{2}}$, or
\begin{align}\label{WZPMQ001}
\frac{|B_{mR}\cap\{u(\cdot,t)<\mu^{-}+2^{-(\kappa_{0}+1)}\omega\}|}{|B_{mR}|}\leq\frac{3}{4},\quad\text{for any }t\in[-R^{p},0],
\end{align}
where $\kappa_{0}$ is given by \eqref{AH26}.

\end{lemma}

\begin{proof}
We prove only \eqref{DQAF001}, since \eqref{WZPMQ001} follows by an analogous argument.\ For $k\in[\mu^{-},\mu^{+}]$, write $A_{m}(k,R)=(B_{mR}\times[-R^{p},0])\cap\{u>k\}$. Let $\zeta\in C^{\infty}_{0}(B_{mR})$ be a cut-off function satisfying that $\zeta=1$ in  $B_{(1-\varrho)mR}$, and $0\leq\zeta\leq1,\, |\nabla\zeta|\leq\frac{2}{\varrho mR}\;\text{in}\; B_{mR}, $ where $\varrho\in(0,1)$ is to be fixed in the following.\ For ease of notation, denote $v_{+}:=(u-(\mu^{+}-2^{-1}\omega))_{+}$. Combining \eqref{ENE01} and \eqref{WMDN001}, we deduce that if $$\omega>\sigma^{-1}\mathcal{M}_{\ast}R^{\frac{p}{2}}\geq\sigma^{-\frac{2-p}{p}}R^{\frac{p}{2}},$$
then
\begin{align}\label{QNE089}
&\sup\limits_{t\in(-R^{p},0)}\int_{B_{(1-\varrho)mR}}v_{+}^{2}dx\notag\\
&\leq\int_{B_{mR}}v_{+}^{2}(x,-R^{p})\zeta^{p}(x)dx+C\int_{B_{mR}\times(-R^{p},0)}v_{+}^{p}|\nabla\zeta|^{p}dxdt\notag\\
&\quad+C\big|A_{m}(\mu^{+}-2^{-1}\omega,R)\big|\notag\\
&\leq\frac{\omega^{2}}{8}|B_{mR}|+C\Big(\frac{\omega^{p}}{(\varrho mR)^{p}}+1\Big)\big|A_{m}(\mu^{+}-2^{-1}\omega,R)\big|\notag\\
&\leq\frac{\omega^{2}}{8}|B_{mR}|+\frac{C\sigma^{2-p}\omega^{2}}{(\varrho R)^{p}}\big|A_{m}(\mu^{+}-2^{-1}\omega,R)\big|\notag\\
&\leq\frac{\omega^{2}}{4}|B_{mR}|\left(\frac{1}{2}+\frac{C\sigma^{2-p}}{\varrho^{p}}\frac{|A_{m}(\mu^{+}-2^{-1}\omega,R)|}{|Q(mR,R^{p})|}\right).
\end{align}

On the other hand, for every $t\in[-R^{p},0]$, we deduce 
\begin{align*}
&\int_{B_{(1-\varrho)mR}}v_{+}^{2}(x,t)dx\notag\\
&\geq\frac{\omega^{2}}{4}(1-2^{-\kappa_{0}})^{2}|B_{(1-\varrho)mR}\cap\{u(\cdot,t)>\mu^{+}-2^{-(\kappa_{0}+1)}\omega\}|,
\end{align*}
where $\kappa_{0}$ is given by \eqref{AH26}. Inserting this into \eqref{QNE089}, we have
\begin{align*}
&|B_{(1-\varrho)mR}\cap\{u(\cdot,t)>\mu^{+}-2^{-(\kappa_{0}+1)}\omega\}|\notag\\
&\leq\frac{|B_{mR}|}{(1-2^{-\kappa_{0}})^{2}}\left(\frac{1}{2}+\frac{C\sigma^{2-p}}{\varrho^{p}}\frac{|A_{m}(\mu^{+}-2^{-1}\omega,R)|}{|Q(mR,R^{p})|}\right).
\end{align*}
Observe that $|B_{mR}\setminus B_{(1-\varrho)mR}|\leq C\varrho|B_{mR}|$. Fix
\begin{align*}
\varrho=\sigma^{\frac{2-p}{p+1}}\bigg(\frac{|A_{m}(\mu^{+}-2^{-1}\omega,R)|}{|Q(mR,R^{p})|}\bigg)^{\frac{1}{p+1}}.
\end{align*}
Therefore, a consequence of these above facts shows that for every $t\in[-R^{p},0]$, 
\begin{align*}
&\frac{|B_{mR}\cap\{u(\cdot,t)>\mu^{+}-2^{-(\kappa_{0}+1)}\omega\}|}{|B_{mR}|}\notag\\
&\leq\frac{1}{2(1-2^{-\kappa_{0}})^{2}}+\frac{C}{\varrho^{p}}\left(\varrho^{p+1}+\sigma^{2-p}\frac{|A_{m}(\mu^{+}-2^{-1}\omega,R)|}{|Q(mR,R^{p})|}\right)\notag\\
&\leq\frac{1}{2(1-2^{-\kappa_{0}})^{2}}+C_{\ast}\sigma^{\frac{2-p}{p+1}}\bigg(\frac{|A_{m}(\mu^{+}-2^{-1}\omega,R)|}{|Q(mR,R^{p})|}\bigg)^{\frac{1}{p+1}}\notag\\
&\leq\frac{1}{2(1-2^{-\kappa_{0}})^{2}}+C_{\ast}\sigma^{\frac{2-p}{p+1}}\leq\frac{5}{8}+\frac{1}{8}=\frac{3}{4},\quad\text{if $0<\sigma\leq\sigma_{0}:=\big(8C_{\ast}\big)^{-\frac{p+1}{2-p}}$}.
\end{align*}
The proof is complete.

\end{proof}

We are now ready to use Corollary \ref{coro008}, Lemmas \ref{LEM0035} and \ref{LEM006} to obtain the improvement on oscillation of $u$ as follows.
\begin{corollary}\label{AMZW01}
The constants $\sigma$ and $A$, and subsequently $a_{0}$ and $m$, can be chosen, and there exists a constant $\kappa_{\ast}>1$, all depending only on the prescribed data, such that

$(i)$ if
\begin{align}\label{W008}
\frac{|B_{mR/2}\cap\{u(\cdot,-(R/2)^{p})>\mu^{+}-2^{-1}\omega\}|}{|B_{mR/2}|}\leq\frac{1}{2},
\end{align}
then one has either $\omega\leq \sigma^{-1}AR^{\frac{p}{2}}=2^{\kappa_{\ast}}\mathcal{M}_{\ast}R^{\frac{p}{2}}$, or
\begin{align}\label{W009}
\sup\limits_{Q(mR/4,(R/4)^{p})}u\leq \mu^{+}-\frac{\omega}{2^{\kappa_{\ast}+1}};
\end{align}

$(ii)$ if
\begin{align*}
\frac{|B_{mR/2}\cap\{u(\cdot,-(R/2)^{p})<\mu^{-}+2^{-1}\omega\}|}{|B_{mR/2}|}\leq\frac{1}{2},
\end{align*}
then one has either $\omega\leq \sigma^{-1}AR^{\frac{p}{2}}=2^{\kappa_{\ast}}\mathcal{M}_{\ast}R^{\frac{p}{2}}$, or
\begin{align}\label{W010}
\inf\limits_{Q(mR/4,(R/4)^{p})}u\geq \mu^{-}+\frac{\omega}{2^{\kappa_{\ast}+1}}.
\end{align}

\end{corollary}
\begin{proof}
We prove only \eqref{W009} in the following, since \eqref{W010} follows by the same argument. First, we can choose a sufficiently large constant $j_{0}>3$, depending only upon the above data, satisfying that
\begin{align}\label{W012}
2^{-(\kappa_{0}+1+j_{0})}\leq\sigma_{0},\quad4\widehat{C}_{0}j_{0}^{-\frac{p-1}{p}}\leq\gamma_{0},
\end{align}
where $\widehat{C}_{0},\,\kappa_{0},\,\gamma_{0}$ and $\sigma_{0}$ are given by Corollary \ref{coro008}, \eqref{AH26}, Lemmas \ref{LEM0035} and \ref{LEM006}, respectively. Then using Lemma \ref{LEM006} with \eqref{W008} and applying Corollary \ref{coro008} with $\delta=\frac{1}{2^{\kappa_{0}+1}}$, $\sigma=\frac{1}{2^{\kappa_{0}+1+j_{0}}}$, $\gamma=\frac{1}{4}$, and $j=j_{0}$, we obtain from \eqref{W012} that either $\omega\leq \sigma^{-1}AR^{\frac{p}{2}}=2^{\kappa_{0}+1+j_{0}}AR^{\frac{p}{2}}$ with $A=\mathcal{M}_{\ast}$,
or
\begin{align}\label{W013}
&\frac{|Q(mR/2,(R/2)^{p})\cap\{u>\mu^{+}-\frac{\omega}{2^{\kappa_{0}+1+j_{0}}}\}|}{|Q(mR/2,(R/2)^{p})|}\leq4\widehat{C}_{0}j_{0}^{-\frac{p-1}{p}}\leq\gamma_{0}.
\end{align}
Utilizing \eqref{W013} and applying Lemma \ref{LEM0035} with $\sigma=\frac{1}{2^{\kappa_{0}+1+j_{0}}}$, we have either $\omega\leq \sigma^{-1}AR^{\frac{p}{2}}=2^{\kappa_{0}+1+j_{0}}\mathcal{M}_{\ast}R^{\frac{p}{2}}$, or
\begin{align*}
u\leq \mu^{+}-\frac{\omega}{2^{\kappa_{0}+2+j_{0}}},\quad\mathrm{in}\text{ $Q(mR/4,(R/4)^{p})$.}
\end{align*}
Consequently, the proof is complete by picking $\kappa_{\ast}:=\kappa_{0}+1+j_{0}$.

\end{proof}

\subsection{The proof of Theorem \ref{ZWTHM90}.}
Based on Corollary \ref{AMZW01}, we establish the following estimates for the oscillation of the solution.
\begin{prop}\label{PRO30}
Given $R\in(0,\frac{1}{2}]$, we have the following alternative:\ either $\omega\leq 2^{\kappa_{\ast}}AR^{\frac{p}{2}}$, or
\begin{align}\label{N90}
\mathop{osc}\limits_{Q(mR/4,(R/4)^{p})}u\leq\eta_{\ast}\omega,\quad\eta_{\ast}=1-2^{-(\kappa_{\ast}+1)},
\end{align}
with \(\kappa_\ast\) and \(A\) as specified in Corollary \ref{AMZW01}.
\end{prop}
\begin{proof}
To begin with, one of the following two alternatives must occur:
\begin{align}\label{WAMQ001}
|B_{mR/2}\cap\{u(\cdot,-(R/2)^{p})>\mu^{+}-2^{-1}\omega\}|\leq\frac{1}{2}|B_{mR/2}|,
\end{align}
and
\begin{align}\label{WAMQ002}
|B_{mR/2}\cap\{u(\cdot,-(R/2)^{p})<\mu^{-}+2^{-1}\omega\}|\leq\frac{1}{2}|B_{mR/2}|.
\end{align}
Then we obtain from Corollary \ref{AMZW01} that when $\omega>2^{\kappa_{\ast}}AR^{\frac{p}{2}}$,
\begin{align*}
\sup\limits_{Q(mR/4,(R/4)^{p})}u\leq\mu^{+}-2^{-(\kappa_{\ast}+1)}\omega,\quad\text{if \eqref{WAMQ001} is satisfied,}
\end{align*}
and
\begin{align*}
\inf\limits_{Q(mR/4,(R/4)^{p})}u\geq\mu^{-}+2^{-(\kappa_{\ast}+1)}\omega,\quad\text{if \eqref{WAMQ002} is satisfied.}
\end{align*}
In both cases, we obtain
\begin{align*}
\mathop{osc}\limits_{Q(mR/4,(R/4)^{p})}u\leq\big(1-2^{-(k_{\ast}+1)}\big)\omega.
\end{align*}
The proof is finished.

\end{proof}

Building on Proposition \ref{PRO30}, in the following we construct a nested family of shrinking cylinders with a common vertex such that the essential oscillation of $u$ over these cylinders tends to zero as their radii tend to zero.

Write
\begin{align}\label{K99}
A_{\ast}:=2^{\kappa_{\ast}}A=2^{\kappa_{\ast}}\mathcal{M}_{\ast},\quad \omega_{0}:=\max\{\omega,A_{\ast}R^{\frac{p}{2}}\},
\end{align}
and, for $k\geq0$,
\begin{align}\label{K98}
R_{k}:=A_{\ast}^{-k}R,\quad \omega_{k+1}:=\max\big\{\eta_{\ast}\omega_{k},A_{\ast}R_{k}^{\frac{p}{2}}\big\},\quad \tilde{a}_{k}:=\Big(\frac{\omega_{k}}{A_{\ast}}\Big)^{\frac{p-2}{p}}.
\end{align}
Combining \eqref{K99} and \eqref{K98}, we deduce
\begin{align*}
\tilde{a}_{k+1}R_{k+1}=&\Big(\frac{\omega_{k+1}}{A_{\ast}}\Big)^{\frac{p-2}{p}}\frac{R_{k}}{A_{\ast}}\leq\frac{\tilde{a}_{k}R_{k}}{\eta_{\ast}^{\frac{2-p}{p}}A_{\ast}}< \tilde{a}_{k}R_{k}.
\end{align*}
Then we have $Q(\tilde{a}_{k+1}R_{k+1},R_{k+1}^{p})\subset Q(\tilde{a}_{k}R_{k},R_{k}^{p})$. 
\begin{lemma}\label{LEM83}
For any $k\geq0$, we obtain
\begin{align}\label{OSC98}
\mathop{osc}\limits_{Q(\tilde{a}_{k}R_{k},R_{k}^{p})}u\leq\omega_{k}.
\end{align}

\end{lemma}

\begin{proof}
It follows from \eqref{K98} that $\tilde{a}_{0}\leq a_{0}$, and hence \eqref{OSC98} holds trivially for $k=0$. Suppose that \eqref{OSC98} holds for some  $k=i$, where $i\geq1$. We show that it then also holds for $k=i+1$. By the induction hypothesis, we know $\mathop{osc}\limits_{Q(\tilde{a}_{i}R_{i},R_{i}^{p})}u\leq\omega_{i}$. Then a direct application of the proof of Proposition \ref{PRO30} with minor modification gives that
\begin{align}\label{AEM81}
\mathop{osc}\limits_{Q(m_{i}R_{i}/4,(R_{i}/4)^{p})}u\leq\max\{\eta_{\ast}\omega_{i},A_{\ast}R_{i}^{\frac{p}{2}}\}=\omega_{i+1},\quad m_{i}:=\Big(\frac{\omega_{i}}{A_{\ast}}\Big)^{\frac{p-2}{p}}.
\end{align}
In light of the fact that $\omega_{i+1}\geq\eta_{\ast}\omega_{i}$ and $1<p<2$, we deduce
\begin{align*}
&\frac{m_{i}R_{i}}{4}=\Big(\frac{\omega_{i}}{A_{\ast}}\Big)^{\frac{p-2}{p}}\frac{R_{i}}{4}\geq \Big(\frac{\omega_{i+1}}{A_{\ast}}\Big)^{\frac{p-2}{p}}R_{i+1}\frac{A_{\ast}^{\frac{2p-2}{p}}(\eta_{\ast}A_{\ast})^{\frac{2-p}{p}}}{4}\geq \tilde{a}_{i+1}R_{i+1}.
\end{align*}
Substituting this into \eqref{AEM81}, we have
\begin{align*}
\mathop{osc}\limits_{Q(\tilde{a}_{i+1}R_{i+1},R_{i+1}^{p})}u\leq\omega_{i+1}.
\end{align*}
The proof is complete.
\end{proof}

For $1<p<2$, define
\begin{align}\label{AH29}
	\varepsilon_{\ast}=\min\Big\{\frac{p}{2},-\frac{\ln\eta_{\ast}}{\ln A_{\ast}}\Big\},
\end{align}
where $\eta_{\ast}$ and $A_{\ast}$ are, respectively, defined by \eqref{N90} and \eqref{K99}. Lemma \ref{LEM83} further yields the following oscillation estimates, exhibiting a quantitative decay rate.
\begin{lemma}\label{PRO90}
Let $1<p<2$. For any $0<\rho\leq R\leq2^{-2/p}$, we have
\begin{align}\label{AH30}
\mathop{osc}\limits_{Q(\tilde{a}_{0}\rho,\rho^{p})}u\leq \max\big\{\eta_{\ast}^{-1}\omega_{0},A_{\ast}^{1+2\varepsilon_{\ast}}R^{\varepsilon_{\ast}}\big\}\left(\frac{\rho}{R}\right)^{\varepsilon_{\ast}},
\end{align}
where $\omega_{0},\tilde{a}_{0}$ and $\varepsilon_{\ast}$ are given by \eqref{K99}--\eqref{K98} and \eqref{AH29}, respectively.

\end{lemma}

\begin{remark}\label{REM10}
By translating the coordinates and slightly modifying the argument of Lemma \ref{PRO90}, we obtain that for any $0<\rho\leq R\leq2^{-2/p}$, $x_{0}\in \overline{B}_{\frac{1}{2}}$, and $t_{0}\in[-\frac{1}{2},0]$,
\begin{align*}
\mathop{osc}\limits_{[(x_{0},t_{0})+Q(\tilde{a}_{0}\rho,\rho^{p})]}u\leq \max\big\{\eta_{\ast}^{-1}\omega_{0},A_{\ast}^{1+2\varepsilon_{\ast}}R^{\varepsilon_{\ast}}\big\}\left(\frac{\rho}{R}\right)^{\varepsilon_{\ast}}.
\end{align*}
\end{remark}

\begin{proof}
Given $0<\rho\leq R\leq2^{-2/p}$, one can choose an integer $k\geq0$ satisfying that $R_{k+1}=A_{\ast}^{-(k+1)}R\leq\rho\leq A_{\ast}^{-k}R=R_{k}$. In the case of $k=0$, it follows from \eqref{AH29} that
\begin{align*}
\eta_{\ast}^{-1}\Big(\frac{\rho}{R}\Big)^{\varepsilon_{\ast}}\geq\eta_{\ast}^{-1}A_{\ast}^{-\varepsilon_{\ast}}\geq1,
\end{align*}	
which, in combination with Lemma \ref{LEM83}, shows that \eqref{AH30} holds for $k=0.$ Hence, it remains only to consider the case of $k\geq1$. Making use of \eqref{AH29} and $A_{\ast}^{-(k+1)}R\leq\rho\leq A_{\ast}^{-k}R$, we obtain
	\begin{align}\label{D31}
		\omega_{k}\leq&\max\{\eta_{\ast}\omega_{k-1},A_{\ast}R_{k-1}^{\varepsilon_{\ast}}\}\leq\max\big\{\eta_{\ast}^{k}\omega_{0},\max\limits_{0\leq i\leq k-1}A_{\ast}\eta_{\ast}^{i}R_{k-1-i}^{\varepsilon_{\ast}}\big\}\notag\\
		\leq&\max\big\{\eta_{\ast}^{-1}A_{\ast}^{-(k+1)\varepsilon_{\ast}}\omega_{0},\max\limits_{0\leq i\leq k-1}(\eta_{\ast}A_{\ast}^{\varepsilon_{\ast}})^{i}R^{\varepsilon_{\ast}}A_{\ast}^{1-(k-1)\varepsilon_{\ast}}\big\}\notag\\
		\leq&\max\bigg\{\eta_{\ast}^{-1}\left(\frac{\rho}{R}\right)^{\varepsilon_{\ast}}\omega_{0},R^{\varepsilon_{\ast}}A_{\ast}^{1-(k-1)\varepsilon_{\ast}}\bigg\}\notag\\
		\leq&\max\big\{\eta_{\ast}^{-1}\omega_{0},A_{\ast}^{1+2\varepsilon_{\ast}}R^{\varepsilon_{\ast}}\big\}\left(\frac{\rho}{R}\right)^{\varepsilon_{\ast}}.
	\end{align}
From \eqref{K99} and \eqref{K98}, we have
	\begin{align*}
		\omega_{k}=&\max\big\{\eta_{\ast}\omega_{k-1},A_{\ast}R_{k-1}^{\frac{p}{2}}\big\}=\max\big\{\eta_{\ast}^{k}\omega_{0},\max\limits_{0\leq i\leq k-1}A_{\ast}\eta_{\ast}^{i}R_{k-1-i}^{\frac{p}{2}}\big\}\notag\\
		=&\max\big\{\eta_{\ast}^{k}\omega_{0},A_{\ast}R^{\frac{p}{2}}\max\limits_{0\leq i\leq k-1}\eta_{\ast}^{i}A_{\ast}^{-\frac{p(k-1-i)}{2}}\big\}\leq\omega_{0},
	\end{align*}
which, together with \eqref{D31}, gives that
	\begin{align*}
		\mathop{osc}\limits_{Q(\tilde{a}_{0}\rho,\rho^{p})}u\leq\mathop{osc}\limits_{Q(\tilde{a}_{k}R_{k},R_{k}^{p})}u\leq \max\big\{\eta_{\ast}^{-1}\omega_{0},A_{\ast}^{1+2\varepsilon_{\ast}}R^{\varepsilon_{\ast}}\big\}\left(\frac{\rho}{R}\right)^{\varepsilon_{\ast}}.
	\end{align*}

\end{proof}

Before giving the proof of Theorem \ref{ZWTHM90}, we first list a geometric covering lemma. For $x,y\in\mathbb{R}^{n}$, we denote by
\begin{align*}
	\mathrm{seg}[x,y]:=\set{(1-\tau)x+\tau y:\tau\in[0,1]}
\end{align*}
the closed line segment joining $x$ and $y$.

\begin{lemma}\label{Lem-COVER01}
Let $1<p<2$ and $x,y\in B_{2^{-1}+2^{-2/p}}\subset B_{1}$.\ Then there exist points $x_1,x_2,x_3\in B_{1/2}$ and $z_1,z_2\in\mathrm{seg}[x,y]$ such that
\begin{align*}
	\mathrm{seg}[x,z_1]
	&\subset B_{2^{-2/p}}(x_1)\subset B_1,\quad\mathrm{seg}[z_1,z_2]\subset B_{2^{-2/p}}(x_2)\subset B_1,\\
	\mathrm{seg}[z_2,y]
	&\subset B_{2^{-2/p}}(x_3)\subset B_1.
\end{align*}
\end{lemma}
For readers' convenience, we provide the proof in the appendix.

\begin{proof}[Proof of Theorem \ref{ZWTHM90}]
{\bf Step 1.} Since $1<p<2$, we obtain from \eqref{M02} and \eqref{K99} that for every $0<R\leq\frac{1}{2}$,
\begin{align*}
\tilde{a}_{0}=\left(\frac{A_{\ast}}{\max\{\omega,A_{\ast}R^{\frac{p}{2}}\}}\right)^{\frac{2-p}{p}}\geq \left(\frac{A_{\ast}}{\max\{\mathcal{M}_{\ast},A_{\ast}2^{-\frac{p}{2}}\}}\right)^{\frac{2-p}{p}}\geq1.
\end{align*}
Therefore, Theorem \ref{CORO06}, Lemma \ref{PRO90}, and Remark \ref{REM10} together imply that there exists a constant $\alpha\in(0,\frac{p}{2}]$, depending only on the prescribed data, such that for any $t_{0}\in(-\frac{1}{2},0]$, $x_{0}\in \overline{B}_{\frac{1}{2}}$, and $\rho\in(0,2^{-2/p}]$,
\begin{align*}
\mathop{osc}\limits_{[(x_{0},t_{0})+Q(\rho,\rho^{p})]}u\leq C\rho^{\alpha}.
\end{align*}
This, in combination with the time-Lipschitz estimate in \eqref{AH01}, leads to that for any $t_{0}\in(-\frac{1}{2},0]$, $x_{0}\in \overline{B}_{\frac{1}{2}}$ and $(x,t)\in B_{2^{-2/p}}(x_{0})\times(-\frac{3}{4},t_{0}]$,

$(1)$ for $|t-t_{0}|\leq\frac{1}{4}$, 
\begin{align}\label{AH32}
|u(x,t)-u(x_{0},t_{0})|\leq&|u(x,t)-u(x,t_{0})|+|u(x,t_{0})-u(x_{0},t_{0})|\notag\\
\leq& \frac{\mathcal{Q}}{2-p}\frac{t_{0}-t}{t_{0}+1}+C|x-x_{0}|^{\alpha}\leq C\big(|x-x_{0}|^{\alpha}+|t-t_{0}|\big);
\end{align}

$(2)$ for $|t-t_{0}|>\frac{1}{4}$, one can select two time points $\{t_{i}\}_{i=1}^{2}$ such that $t<t_{1}\leq t_{2}<t_{0}$,
\begin{align}\label{AH33}
|u(x,t)-u(x_{0},t_{0})|\leq&|u(x,t)-u(x,t_{1})|+|u(x,t_{1})-u(x,t_{2})|\notag\\
&+|u(x,t_{2})-u(x,t_{0})|+|u(x,t_{0})-u(x_{0},t_{0})|\notag\\
\leq& \frac{\mathcal{Q}}{2-p}\bigg(\frac{t_{1}-t}{t_{1}+1}+\frac{t_{2}-t_{1}}{t_{2}+1}+\frac{t_{0}-t_{2}}{t_{0}+1}\bigg)+C|x-x_{0}|^{\alpha}\notag\\
\leq&C\big(|x-x_{0}|^{\alpha}+|t-t_{0}|\big),
\end{align}
where $\mathcal{Q}$ is given by \eqref{DK02}.

{\bf Step 2.} Let $x_{0}\in\overline{B}_{\frac{1}{2}}.$ In light of \eqref{DOM01}, we fix any two points $(x,t),(\hat{x},\hat{t})\in B_{2^{-2/p}}(x_{0})\times(-\frac{1}{2},0].$ We may assume without loss of generality that $|\hat{x}-x_{0}|\leq|x-x_{0}|$. Set $R=|x-x_{0}|$.  For each $(y,s)\in Q(1/R,1/R^{p})$, introduce the rescaled functions as follows:
\begin{align*}
u_{R}(y,s)=u(Ry+x_{0},R^{p}s),\quad f_{R}(y)=R^{p}f(Ry+x_{0}).
\end{align*}
Then $u_{R}$ solves
\begin{align*}
\partial_{s}u_{R}-\mathrm{div}(|\nabla u_{R}|^{p-2}\nabla u_{R})=f_{R}(y),\quad \mathrm{in}\;Q(1/(2^{2/p}R),1/R^{p}).
\end{align*}
Adapting the argument of \eqref{AH32}--\eqref{AH33} slightly, we find constants $0<\beta<\frac{p}{2}$ and $C>1$, depending only on the data, such that for any fixed $\bar{y}\in \partial B_{1}$ and $\bar{s}\in(-2^{-1} R^{-p},0]$,
\begin{align}\label{WAQA001}
|u_{R}(y,s)-u_{R}(\bar{y},\bar{s})|\leq C\big(|y-\bar{y}|^{\beta}+|s-\bar{s}|\big),
\end{align}
for each $(y,s)$ with $|y-\bar{y}|+|s-\bar{s}|^{1/p}\leq2^{-2/p}$. We may henceforth assume that $\beta\in(0,\alpha]$, since if $\beta>\alpha$, estimate \eqref{WAQA001} also holds with any smaller exponent in $(0,\alpha]$.

Using \eqref{AH01}, we have
\begin{align}\label{QT01}
|u(x,t)-u(\hat{x},\hat{t})|\leq&|u(x,t)-u(x,\hat{t})|+|u(x,\hat{t})-u(\hat{x},\hat{t})|\notag\\
\leq& \frac{2\mathcal{Q}}{2-p}|t-\hat{t}|+|u(x,\hat{t})-u(\hat{x},\hat{t})|.
\end{align}
We next handle the second term in \eqref{QT01}. Let $c\geq2$. In the case when $|x-\hat{x}|\leq R^{c}$, it follows from \eqref{WAQA001} that
\begin{align*}
|u(x,\hat{t})-u(\hat{x},\hat{t})|=&\left|u_{R}((x-x_{0})/R,\hat{t}/R^{p})-u_{R}((\hat{x}-x_{0})/R,\hat{t}/R^{p})\right|\notag\\
\leq& C|(x-\hat{x})/R|^{\beta}\leq C|x-\hat{x}|^{\frac{(c-1)\beta}{c}}.
\end{align*}
On the other hand, when $|x-\hat{x}|>R^{c}$, we have from \eqref{AH32} that
\begin{align*}
|u(x,\hat{t})-u(\hat{x},\hat{t})|\leq&|u(x,\hat{t})-u(x_{0},\hat{t})|+|u(x_{0},\hat{t})-u(\hat{x},\hat{t})|\notag\\
\leq& C(R^{\alpha}+|\hat{x}-x_{0}|^{\alpha})\leq C|x-\hat{x}|^{\frac{\alpha}{c}}.
\end{align*}
Then by picking $c=1+\frac{\alpha}{\beta}$, we obtain that for any $(x,t),(\hat{x},\hat{t})\in B_{2^{-2/p}}(x_{0})\times(-\frac{1}{2},0],$
\begin{align}\label{AH36}
|u(x,t)-u(\hat{x},\hat{t})|\leq C\big(|x-\hat{x}|^{\frac{\alpha\beta}{\alpha+\beta}}+|t-\hat{t}|\big).
\end{align}

{\bf Step 3.} Using \eqref{DOM01}, we pick any two points $(x,t),(y,s)\in B_{2^{-1}+2^{-2/p}}\times(-\frac{1}{2},0]$. It then follows from Lemma \ref{Lem-COVER01} that one can find $x_1,x_2,x_{3}\in\overline{B}_{\frac{1}{2}}$ and $z_{1},z_{2}\in\mathrm{seg}[x,y]$ such that
\begin{align}\label{AH37}
	\begin{cases}
	\mathrm{seg}[x,z_1]
	\subset B_{2^{-2/p}}(x_1)\subset B_1,\quad\mathrm{seg}[z_1,z_2]\subset B_{2^{-2/p}}(x_2)\subset B_1,\\
	\mathrm{seg}[z_2,y]
	\subset B_{2^{-2/p}}(x_3)\subset B_1.
	\end{cases}
\end{align}
Hence, a combination of \eqref{AH36}--\eqref{AH37} leads to that
\begin{align*}
|u(x,t)-u(y,s)|\leq& |u(x,t)-u(z_{1},t)|+|u(z_{1},t)-u(z_{2},t)|+|u(z_{2},t)-u(y,s)|\notag\\
\leq& C\big(|x-z_{1}|^{\frac{\alpha\beta}{\alpha+\beta}}+|z_{1}-z_{2}|^{\frac{\alpha\beta}{\alpha+\beta}}+|z_{2}-y|^{\frac{\alpha\beta}{\alpha+\beta}}+|t-s|\big)\notag\\
\leq &C\big(|x-y|^{\frac{\alpha\beta}{\alpha+\beta}}+|t-s|\big).	
\end{align*}	
The proof is complete.

\end{proof}

\section{Appendix:\,The proof of Lemma \ref{Lem-COVER01}}

\begin{proof}[Proof of Lemma \ref{Lem-COVER01}]
	Set
\begin{align*}
	r=2^{-2/p},\quad
	R=\frac{1}{2}+r,\quad
	\theta=\frac{1}{2R}=\frac{1}{1+2r},\quad
	\Lambda=\max\{|x|,|y|\}.
\end{align*}
Since $1<p<2$ and $x,y\in B_{2^{-1}+2^{-2/p}}$, we have
\begin{align*}
	\frac{1}{4}<r<\frac{1}{2},\quad
	\frac{1}{2}<\theta<\frac{2}{3},\quad
	\Lambda<R.
\end{align*}
Define
\begin{align*}
	x_1=\theta x,\quad
	x_2=\frac{\theta}{2}(x+y),\quad
	x_3=\theta y,
\end{align*}
and
\begin{align*}
	z_1=\frac{2x+y}{3},\quad
	z_2=\frac{x+2y}{3}.
\end{align*}
Clearly, $z_1,z_2$ divide $\mathrm{seg}[x,y]$ into three
equal subsegments. Moreover,
\begin{align*}
	|x_i|\leq\theta \Lambda<\theta R=\frac{1}{2},
	\quad i=1,2,3.
\end{align*}
Thus $x_i\in B_{1/2}$. For every $\xi\in B_r(x_i)$,
\begin{align*}
	|\xi|\leq|\xi-x_i|+|x_i|
	<r+\frac12=R<1,
\end{align*}
so $B_r(x_i)\subset B_R\subset B_1$.

It remains to verify that the endpoints of each subsegment
belong to the corresponding ball. We have
\begin{align*}
	x-x_1&=(1-\theta)x,\quad z_1-x_1=\left(\frac23-\theta\right)x+\frac13y,\\
	z_1-x_2&=\left(\frac23-\frac{\theta}{2}\right)x
	+\left(\frac13-\frac{\theta}{2}\right)y,\\
	z_2-x_2&=\left(\frac13-\frac{\theta}{2}\right)x
	+\left(\frac23-\frac{\theta}{2}\right)y,\\
	z_2-x_3&=\frac13x+\left(\frac23-\theta\right)y,\quad y-x_3=(1-\theta)y.
\end{align*}
Since $\theta<2/3$, each expression on the right-hand side
is a nonnegative linear combination of $x$ and $y$ whose
coefficients sum to $1-\theta$. Consequently,
\begin{align*}
	&\max\big\{
	|x-x_1|,\,
	|z_1-x_1|,\,
	|z_1-x_2|,\,
	|z_2-x_2|,\,
	|z_2-x_3|,\,
	|y-x_3|
	\big\}\notag\\
&\leq(1-\theta)\Lambda
	<(1-\theta)R=r.
\end{align*}
The desired segment inclusions now follow from the convexity of the balls.
\end{proof}

%
%%
%\noindent{\bf{\large Conflict of interest.}} The authors declare that they have no conflict of interest.
%\noindent{\bf{\large Data Availability Statement.}} The data used to support the findings of this study are available from the corresponding author upon request.

\noindent{\bf{\large Acknowledgements.}} X. Hao was partially supported by Science and Technology Project of Hebei Education Department (No.\ QN2024074) and Hebei Natural Science Foundation (No.\ A2026205019).\ Z.W. Zhao was partially supported by NSFC (No.\ 12501254).

%\noindent{\bf{\large Statements.}}

%\noindent{\bf{\large Acknowledgements.}}

\end{document}